\documentclass[a4paper,12pt]{article}

\usepackage{bbm}
\usepackage{amsmath, wrapfig}
\usepackage{dsfont, a4wide, amsthm, amssymb, amsfonts, graphicx}
\usepackage{fancyhdr, xspace, psfrag, setspace, supertabular, color, enumerate}
\usepackage{hyperref}

\newtheorem{theorem}{Theorem}[section]
\newtheorem{lemma}[theorem]{Lemma}
\newtheorem{proposition}[theorem]{Proposition}
\newtheorem{corollary}[theorem]{Corollary}
\newtheorem{conjecture}[theorem]{Conjecture}
\newtheorem*{conjecture*}{Conjecture}

\newtheorem*{theorem*}{Theorem}
\newtheorem{definition}[theorem]{Definition}
\newtheorem{question}[theorem]{Question}
\newtheorem*{question*}{Question}

\theoremstyle{definition}
\newtheorem{example}{Example}
\newtheorem{remark}[theorem]{Remark}

\def\e{\epsilon}
\def\a{\alpha}

\def\g{\gamma}
\def\d{\delta}
\def\ce{\mathcal{E}}
\def\cf{\mathcal{F}}
\def\cg{\mathcal{G}}

\def\cc{\mathcal{C}}
\def\cd{\mathcal{D}}

\def\N{\mathbb{N}}
\def\Z{\mathbb{Z}}
\def\E{\mathop{\mathbb{E}}}
\def\P{\mathbb{P}}
\def\C{\mathbb{C}}
\def\F{\mathbb{F}}

\def\so3{\mathrm{SO}(3)}
\def\fso3{\mathrm{FSO}(3)}

\DeclareMathOperator{\rk}{rk}

\title{Low-complexity approximations for sets defined by generalizations of affine conditions}

\author{W. T. Gowers\footnote{Coll\`ege de France and Department of Pure Mathematics and Mathematical Statistics, University of Cambridge. Email: \texttt{wtg10@dpmms.cam.ac.uk}.} \and Thomas Karam\footnote{Mathematical Institute, University of Oxford. This research was essentially carried out when that author was at the Department of Pure Mathematics and Mathematical Sciences, University of Cambridge. Email: \texttt{thomas.karam@maths.ox.ac.uk}.}}
\begin{document}
\maketitle
%\section{}
%\subsection{}

\begin{abstract}
Let $p$ be a prime, let $S$ be a non-empty subset of $\F_p$ and let $0<\epsilon\leq 1$. We show that there exists a constant $C=C(p, \epsilon)$ such that for every positive integer $k$, whenever $\phi_1, \dots, \phi_k: \F_p^n \rightarrow \F_p$ are linear forms and $E_1, \dots, E_k$ are subsets of $\F_p$, there exist linear forms $\psi_1, \dots, \psi_C: \F_p^n \rightarrow \F_p$ and subsets $F_1, \dots, F_C$ of $\F_p$ such that the set $U=\{x \in S^n: \psi_1(x) \in F_1, \dots, \psi_C(x) \in F_C\}$ is contained inside the set $V=\{x \in S^n: \phi_1(x) \in E_1, \dots, \phi_k(x) \in E_k\}$, and the difference $V \setminus U$ has density at most $\epsilon$ inside $S^n$. We then generalize this result to one where $\phi_1, \dots, \phi_k$ are replaced by homomorphisms $G^n \to H$ for some pair of finite Abelian groups $G$ and $H$, and to another where they are replaced by polynomial maps $\F_p^n \to \F_p$ of small degree.

 \end{abstract}

\tableofcontents

\section{Introduction}\label{Section: Introduction}

Let $p$ be a prime. If $\phi_i(x) = b_i$, $i=1, \dots, k$ is a system of equations over some finite-dimensional vector space over $\F_p$, then the solution set is either empty or an affine subspace with codimension equal to $\dim \langle \phi_1, \dots, \phi_k \rangle$, and therefore with density equal to
\[p^{-\dim \langle \phi_1, \dots, \phi_k \rangle}.\] 
It follows that if the density of the set of solutions is at least $\d$, then there is a set of at most $\log_p\d^{-1}$ equations with the same solution set.

In this paper we shall generalize this simple observation in a number of ways. Our initial motivation for doing this was that in the context of results such as the density Hales-Jewett theorem \cite{Furstenberg Katznelson, polymath} certain generalizations of affine subspaces occur naturally. Consider, for instance, the problem of understanding triples of dense subsets $A, B$ and $C$ of $\{0,1,2\}^n$ such that the set $A\times B\times C$ contains no combinatorial line. (One might wish to do this with a view to obtaining a density increment argument.) One large source of examples is constructed as follows. Pick a prime $p$ and three subsets $U,V$ and $W$ of $\F_p$ such that $U\times V\times W$ contains no arithmetic progression (even degenerate), let $\phi:\{0,1,2\}\to\F_p$ be the function where $\phi(x)$ is the mod-$p$ sum of the coordinates of $x$, and let $A=\phi^{-1}(U), B=\phi^{-1}(V)$ and $C=\phi^{-1}(W)$. Since $\phi$ takes combinatorial lines to arithmetic progressions in $\Z$, and therefore in $\F_p$, it follows that $A\times B\times C$ has the desired property. 

The above construction can be generalized further, and led us to consider sets of the following form. Let $G$ be a finite Abelian group, let $S$ be a subset of $G$, let $\phi_1,\dots,\phi_k:G^n\to G$ be homomorphisms, and let $E_1,\dots,E_k$ be subsets of $G$. Then we take the subset
\[\{x\in G^n:\phi_i(x)\in E_i, i=1,2,\dots,k\},\]
which is subspace-like, in that it is the intersection of sets defined by linear-like conditions. One of our main results will be that, as in the simple case where $G=\F_p$ and $E_1,\dots,E_k$ are singletons, if the above subset is dense, then we may replace it by a subset defined by a number of conditions that is bounded above in terms of the density. However, the new subset is not identical to the old one: in this set-up, the best one can hope for is a good approximation, and that is what we shall obtain. The precise statement will be given later. 

When $G=\F_p$, we also prove a generalization in a different direction, considering subsets of $S^n$ that are defined by low-degree polynomial conditions on $\F_p^n$. Again we shall show that if such a subset is dense, then it can be approximated by a subset of bounded complexity, meaning that it is defined using a bounded number of low-degree polynomial conditions. Again, we defer the precise statement till later.

The following examples indicate some of the complexities that arise when we restrict linear conditions to subsets of the form $S^n$. For brevity, if $\phi:G^n\to G$ and $E\subset G$, we write $(\phi,E)$ for the condition $\phi(x)\in E$. 

\begin{example} \label{First example}
For $i=2,3,\dots,k$, let $\phi_i:\{0,1\}^n\to\F_p$ be the linear form defined by $\phi_i(x)=x_1+x_i$ and let $E_i\subset\F_p$ be the set $\{0,1\}$. Then the set of points of $\{0,1\}^n$ satisfying the conditions $(\phi_i,E_i)$ consists of all $x\in\{0,1\}^n$ such that either $x_1=0$, or $x_1=1$ and $x_2=\dots=x_k=0$. The conditions are linearly independent, and the density of this set is $1/2+1/2^k$. Given $\e>0$, either $k\leq\log_2(\e^{-1})$ or we can approximate the set to within density $\e$ by the set of points satisfying the single condition $x_1=0$. 
\end{example} 

\begin{example} \label{Second example}
For $i=2,3,\dots,k$, let $\phi_i:\Z_p^n \to \Z_p^2$ be the group homomorphism defined by $\phi_i(x)=(x_1,x_i)$ and let $E_i\subset\Z_p^2$ be the set $(\{0\} \times \Z_p) \cup \{(1,1)\}$. Then the set of points of $\Z_p^n$ satisfying the conditions $(\phi_i,E_i)$ consists of all $x\in \Z_p^n$ such that either $x_1=0$, or $x_1=x_2=\dots=x_k=1$. The density of this set is $1/p+1/p^k$, so it can be approximated in a similar way.
\end{example}  

\begin{example} \label{Third example}
For $i=2,3,\dots,k$, let $P:\F_p^n \to \F_p$ be the degree-$2$ polynomial defined by $P(x)=x_1x_i$ and let $E_i\subset\F_p$ be the set $\{0\}$. Then the set of points of $\F_p^n$ satisfying the conditions $(\phi_i,E_i)$ consists of all $x\in \F_p^n$ such that either $x_1=0$, or $x_1 \neq 0$ and $x_2=\dots=x_k=0$. The density of this set is $1/p+ (p-1)/p^k$, so again it can be approximated in a similar way.
\end{example}  

Note that in each of the three examples, the low-complexity set that approximates the solution set -- that is, the set of points that satisfy the conditions -- is a subset of the solution set. Our main results will show that this is a general phenomenon. To formulate them it will be useful to have a few definitions.

\begin{definition}
Let $X$ be a finite set and let $U,V$ be subsets of $X$. Then $U$ \emph{internally $\e$-approximates} $V$ if $U\subset V$ and $|V\setminus U|$ has density at most $\e$ in $X$. Given a set $Z$, a function $f:X\to Z$, and a subset $E\subset Z$, we call the pair $(f,E)$ a \emph{condition} on $X$, and we say that $x$ \emph{satisfies the condition} $(f,E)$ if $f(x)\in E$. If $\cf=\{(f_1,E_1),\dots,(f_k,E_k)\}$ is a set of conditions on $X$, the \emph{satisfying set} of the conditions of $\cf$ is the set of $x\in X$ that satisfy all the conditions in $\cf$. If $\cf$ and $\cg$ are sets of conditions on $X$, we say that $\cg$ \emph{internally $\e$-approximates} $\cf$ if the satisfying set of $\cg$ internally $\e$-approximates the satisfying set of $\cf$. If $S\subset\F_p$, then a \emph{mod-$p$ condition} on $S^n$ is a condition $(\phi,E)$ on $S^n$, where $\phi$ is the restriction to $S^n$ of a linear form defined on $\F_p^n$ and $E$ is a subset of $\F_p$. 
\end{definition}

Our first main result is the following theorem. 
\begin{theorem}\label{main approximation theorem}
For every $\e>0$, every prime $p$ and every non-empty $S \subset \F_p$ there exists $K$ such that for every $n$, every set of mod-$p$ conditions on $S^n$ can be internally $\e$-approximated by a set of at most $K$ mod-$p$ conditions. 
\end{theorem}

In order to give some idea of why this might be an interesting statement, we make a few remarks. Note first that if we set $p=2$, then the result is easy. We can identify $\{0,1\}^n$ with $\F_2^n$, and the satisfying set will take the form $\bigcap_{i=1}^r\{x:\phi_i(x)\in E_i\}$, where each $\phi_i$ is a linear form on $\F_2^n$ and each $E_i$ is a subset of $\F_2$. If $E_i=\F_2$, then we can discard the condition $(\phi_i,E_i)$ without affecting the satisfying set, and if any $E_i$ is empty then the satisfying set is empty, so we can obtain it using just the condition $(\phi_i,E_i)$. So without loss of generality all the sets $E_i$ are singletons and the satisfying set is an affine subspace. If the codimension of the subspace is greater than $\log_2(\e^{-1})$ then we can internally $\e$-approximate it by the empty set, which as above just needs one condition. Otherwise, we can pick $k$ linearly independent conditions that specify the subspace and we have a set of at most $\log_2(\e^{-1})$ conditions with the same satisfying set.

It is only slightly harder to prove the following closely related statement about linear conditions on $\F_p^n$ with $p>2$, by which we mean conditions $(\phi,E)$, where $\phi$ is a linear form on $\F_p^n$ and $E\subset\F_p$. However, the bound is slightly weaker, owing to a crude argument at the end of the proof. We have not tried to improve the bound, and indeed it is not obviously possible, because in this case satisfying sets do not have to be affine subspaces.

\begin{proposition}\label{Compression in the basic case} For every $\e>0$ and every $n$, every set of linear conditions on $\F_p^n$ can be internally $\e$-approximated by a set of at most $(p \log \epsilon^{-1})^p$ linear conditions.
\end{proposition}

\begin{proof} Let the conditions be $(\phi_i,E_i)$ for $i=1,\dots,r$. As in the $p=2$ case we may assume that each $E_i$ is a non-empty proper subset of $\F_p$. We may also assume that the $\phi_i$ are distinct, since if $\phi_i=\phi_j$ then we may replace the two conditions $(\phi_i,E_i)$ and $(\phi_j,E_j)$ by the single condition $(\phi_i,E_i\cap E_j)$. 

Now choose a maximal set $I$ such that the forms $\phi_i$ with $i\in I$ are linearly independent. Then if we choose $x\in\F_p^n$ uniformly at random, the events $\phi_i(x)\in E_i$ with $i\in I$ are independent, so the probability that $x$ belongs to the satisfying set is at most $\prod_{i\in I}|E_i|/p$, which is at most $(1-1/p)^{|I|}$. If $|I|>p\log(\e^{-1})$, then this is at most $\e$, so we can internally $\e$-approximate the satisfying set by the empty set, which requires just one condition. Otherwise, the forms $\phi_1,\dots,\phi_r$ are contained in a linear subspace of dimension at most $p\log(\e^{-1})$, and there is at most one condition for each $i$, so $r\leq(p\log(\e^{-1}))^p$, which proves the result.
\end{proof}

The key to the proofs above was that if a set satisfies $k$ non-trivial linearly independent conditions, then its density must be at most $\a^k$ for some $\a<1$ that is independent of $n$. What makes Theorem \ref{main approximation theorem} harder is that this is no longer true for mod-$p$ conditions on $\{0,1\}^n$, as Example \ref{First example} shows. Example \ref{First example} is a simple representative of a large class of examples, but it illustrates the basic challenge that we have to deal with.

Once we have proved Theorem \ref{main approximation theorem}, we shall generalize it to a similar statement with $\F_p$ replaced by an arbitrary finite Abelian group. If $G$ and $H$ are finite Abelian groups and $n$ is a positive integer, then an $H$-\emph{condition} on $G^n$ is a pair $(\phi,E)$, where $\phi:G^n\to H$ is a homomorphism and $E\subset H$. The definition of an internal $\e$-approximation carries over in the obvious way: if $S\subset G$, then we shall say that a set $\cf$ of $H$-conditions on $G^n$ \emph{internally $\e$-approximates} a set $\cg$ of $H$-conditions on $G^n$ if every point that satisfies $\cf$ satisfies $\cg$ and the difference of the satisfying sets has density at most $\e$ inside $S^n$.

\begin{theorem}\label{Compression in the groups case for arbitrary alphabets} Let $G$ and $H$ be finite Abelian groups, let $S$ be a non-empty subset of $G$, and let $\e>0$. Then there exists $A=A(\e,G,S,H)$ such that every set of $H$-conditions can be $\e$-approximated on $S^n$ by a set of at most $A$ $k$-dimensional $H$-conditions. 
\end{theorem}
\noindent This theorem turns out to need additional ideas even when $G=\F_p$ and $H=\F_p^2$.
\medskip

The other generalization we shall prove is to polynomial conditions. If $S\subset\F_p$, then a \emph{$k$-dimensional condition of degree at most $d$} on $S^n$ is a pair $(\phi,E)$, where $\phi:\F_p^n\to\F_p^k$ is a polynomial map of degree at most $d$ (that is, each coordinate of $\phi(x)$ is obtained by applying a polynomial with degree at most $d$ to $x$) and $E\subset\F_p^k$. We prove the following theorem, which has a weaker conclusion than in the linear case, but still shows that a set defined by polynomial conditions can be approximated by a set of bounded complexity defined by polynomial conditions.

\begin{theorem}\label{Compression of polynomial conditions to a single condition} For every $\e>0$, every prime $p$, and every pair of positive integers $k$ and $d$ with $d<p$, there exists $M$ such that for every non-empty $S\subset\F_p$, every set of $k$-dimensional conditions with degree at most $d$ on $S^n$ can be internally $\e$-approximated by a set consisting of a single $m$-dimensional condition with degree at most $d$ on $S^n$ with $m\leq M$.
\end{theorem}

\section{Preliminary lemmas} \label{Section: Equidistribution and lower bound statements on restricted subsets}

The next section will be devoted to a proof of Theorem \ref{main approximation theorem}. To prepare for this, we collect together some more technical statements that will be useful in the proof. They will mainly concern the joint distribution of the values $\phi_1(x),\dots,\phi_k(x)$, where each $\phi_i$ is a linear form from $\F_p^n$ to $\F_p$ and $x$ is chosen uniformly from $S^n$ for some subset $S$ of $\F_p$ of size at least 2. We begin with a result that will help us with the case $k=1$. We write $\omega_p$ for $\exp(2\pi i/p)$.

\begin{lemma}\label{Bound on modulus of ratio}
Let $p$ be a prime and let $S$ be a subset of $\F_p$ that contains at least two elements. Then 
\[|\E_{x \in S} \omega_p^{tx}| \le 1-p^{-2}\] 
for each $t \in \F_p^*$.
\end{lemma}

\begin{proof} By the triangle inequality, it is sufficient to prove the result in the case $|S|=2$, since 
\[\E_{x\in S}\omega_p^{tx}=\E_{{T\subset S}\atop{|T|=2}}\E_{x\in T}\omega_p^{tx}.\]
But if $y\ne 0$, then 
\[|1+\omega_p^y|^2=2+2\cos(2\pi y/p)\leq 2+2\cos(2\pi/p).\]
Using the bound $\cos\theta\leq 1-\theta^2/2+\theta^4/24$, which is at most $1-\theta^2/4$ when $\theta\leq 2\pi/3$, we find that this is at most $4-4\pi^2/4p^2=4-\pi^2/p^2=4(1-\pi^2/4p^2)$. It follows that $|1+\omega_p^y|/2\leq 1-\pi^2/8p^2\leq 1-1/p^2$, which implies the result.
\end{proof}

From this lemma we deduce an upper bound on the Fourier coefficients of a linear phase function restricted to $S^n$. Given a linear form $\phi$ defined by a formula $x\mapsto\sum_{i=1}^n\lambda_ix_i$, we define $Z(\phi)$ to be $\{i:\lambda_i\ne 0\}$. We shall call this set the \emph{support} of $\phi$, even though technically it is the support of the coefficients of $\phi$ with respect to the standard basis. This should cause no confusion, since the complement of the kernel of $\phi$ will play no role in the paper. 

Our bound shows that the Fourier coefficients are exponentially small in the size of the support.

\begin{proposition}\label{Small biases for linear forms} Let $p$ be a prime and let $S$ be a subset of $\F_p$ that contains at least two elements. Let $\phi: \F_p^n \rightarrow \F_p$ be a linear form. Then \[|\E_{x \in S^n} \omega_p^{\phi(x)}| \le (1-p^{-2})^{|Z(\phi)|}.\]  
\end{proposition}

\begin{proof}
If $\phi(x)=\sum_i\lambda_ix_i$, then the left-hand side splits as a product 
\[\E_{x \in S^n} \omega_p^{\phi(x)} = \prod_{i=1}^n \E_{x_i \in S} \omega_p^{\lambda_ix_i}.\] 
The inner expectation of the right-hand side is equal to $1$ for every $i\notin Z(\phi)$, and by Lemma \ref{Bound on modulus of ratio} it has modulus at most $1-p^{-2}$ for every $z \in Z(\phi)$. The result follows. \end{proof}

\begin{remark}
Note that the same conclusion holds for $t\phi$ for every non-zero $t$, since in that case $Z(t\phi)=Z(\phi)$.
\end{remark}

A standard calculation now gives us the equidistribution result we shall need. Given a linear map $\phi:\F_p^n\to\F_p^k$ and an element $a\in\F_p^k$, we write $a.\phi:\F_p^n\to\F_p$ for the map that takes $x$ to $a.\phi(x)$, where $a.b$ is notation for $\sum_{i=1}^ka_ib_i$.

\begin{proposition}\label{Equidistribution of k-tuples of linear forms}
Let $p$ be a prime, let $k$ be a positive integer, and let $S$ be a subset of $\F_p$ that contains at least two elements. Let $\phi: \F_p^n \rightarrow \F_p^k$ be a linear map and let $x$ be chosen uniformly at random from $S^n$. Suppose that $|Z(a.\phi)|\geq r$ for every non-zero $a\in\F_p^k$. Then for each $z\in\F_p^k$, 
\[|\P[\phi(x)=z]-p^{-k}| \le (1-p^{-2})^{r}.\]  
\end{proposition}

\begin{proof}
Let $\d_z$ be the function that takes the value $1$ at $z$ and 0 everywhere else. Then
\[\P[\phi(x)=z]=\E_{x\in S^n}\d_z(\phi(x)),\]
which by the Fourier inversion formula is equal to 
\[\E_{x\in S^n}\sum_{a\in\F_p^k}\hat{\d_z}(a)\omega_p^{a.\phi(x)}=\sum_{a\in\F_p}\hat{\d_z}(a)\E_{x\in S^n}\omega_p^{a.\phi(x)}.\] 
But $\hat{\d_z}(a)=p^{-k}\omega_p^{-a.z}$, so when $t=0$ it is equal to $p^{-k}$, and otherwise it has size at most $p^{-k}$. The result now follows from Proposition \ref{Small biases for linear forms}, the remark after it, and the triangle inequality.
\end{proof}

If a linear form $\phi:\F_p^n\to\F_p$ takes a value once, then it must take that value with probability at least $p^{-1}$. We now show that something like this, but weaker, holds for the restriction of $\phi$ to a set of the form $S^n$.

\begin{proposition}\label{Lower bound on non-zero probabilities} Let $p$ be a prime and let $S$ be a non-empty subset of $\F_p$ that contains at least two elements. For every linear form $\phi: \F_p \rightarrow \F_p$ and every $y \in \F_p$ the probability $\P_{x \in S^n}(\phi(x) = y)$ is either $0$ or at least $|S|^{-t}$, where $t=\lceil(p-1)/(|S|-1)\rceil$. \end{proposition}

\begin{proof}
Let $\phi:\F_p^n\to\F_p$ be a linear form given by the formula $\phi(x)=\sum_{i=1}^n\lambda_ix_i$. Without loss of generality there exists $m$ such that $\lambda_1,\dots,\lambda_m$ are non-zero and $\lambda_{m+1}=\dots=\lambda_n=0$. 

Since there are $|S|^m$ distinct choices for $(x_1,\dots,x_m)$ if $x\in S^n$, each value of $\phi(x)$ that is taken at all is taken with probability at least $|S|^{-m}$. 

We now prove a second upper bound that is less trivial. Write $\lambda.S$ for the dilate $\{\lambda u: u\in S\}$ of $S$ by $\lambda$. By the Cauchy-Davenport theorem, for every $r$ we have the inequality
\[|\lambda_1.S+\lambda_2.S+\dots+\lambda_r.S|\geq \min\{r(|S|-1)+1,p\}.\]
In particular, if $r(|S|-1)\geq p-1$, then $\lambda_1.S+\dots+\lambda_r.S=\F_p$. 

For such $r$, it follows that if $x$ is chosen uniformly at random from $S^n$, then for every $u\in\F_p$, the probability that $\lambda_1x_1+\dots+\lambda_rx_r=u$ is at least $|S|^{-r}$. By conditioning on the value taken by $\lambda_{r+1}x_{r+1}+\dots+\lambda_nx_n$ and applying the law of total probability, we deduce that each element of $\F_p$ is equal to $\phi(x)$ with probability at least $|S|^{-r}$.

Combining this with the first bound gives the result claimed.
\end{proof}

If $S=\{0,1,\dots,h\}$ and $h|(p-1)$, then the bound in Proposition \ref{Lower bound on non-zero probabilities} is optimal, since the linear form $\phi$ defined by $\phi(x) = x_1 + \dots + x_{(p-1)/h}$ takes the value $p-1$ with probability $|S|^{-(p-1)/h}$. Note that when $|S|=2$, the lower bound we obtain for each non-zero probability is $2^{-(p-1)}$. 

\section{Approximating sets of mod-$p$ conditions}\label{Section: Compression of linear conditions in restricted subsets of F_p}

In this section we shall prove Theorem \ref{main approximation theorem}. From now on, we shall write ``$\e$-approximated" to mean ``internally $\e$-approximated", since all the approximations we consider will be internal ones. (Of course, it is then important to keep in mind that the relation ``can be $\e$-approximated by" is not symmetric.) 

Our basic strategy will be to start with a set of mod-$p$ conditions and approximate it by sets with more and more structure until eventually we arrive at a set of bounded size. In order to explain this strategy in more detail, we need some definitions.

\begin{definition} \label{Key structures in the F_p case} Let $p$ be a prime. If $\phi_0:\F_p^n\to\F_p$ is a linear form, then a \emph{sunflower with centre} $\phi_0$ is a set of forms $\{\phi_0+\phi_i:i\in I\}$ such that the forms $\phi_i$ are disjointly supported. A \emph{symmetric family} of forms is a set $\{\phi_i:i\in I\}$ that are equal up to permutations of the coordinates -- that is, for each $y$ the number of coefficients of $\phi_i$ equal to $y$ is the same for each $i$. A \emph{symmetric sunflower} is a sunflower $\{\phi_0+\phi_i:i\in I\}$ such that the set $\{\phi_i:i\in I\}$ is a symmetric family. A set of forms $\{\phi_i:i\in I\}$ is $r$-\emph{separated} if $|Z(\sum_{i\in I}u_i\phi_i)|\geq r$ for every non-zero $u\in\F_p^{I}$ -- that is, if the support size of every non-trivial linear combination of the $\phi_i$ is at least $r$. The \emph{support distance} on the set of all linear forms $\phi:\F_p^n\to\F_p$ is the metric given by the formula $d(\phi,\psi)=|Z(\phi-\psi)|$. The \emph{ball of radius $r$ about $\phi_0$} is the set of all $\phi$ such that $d(\phi_0,\phi)\leq r$. We say that a sunflower $\{\phi_0+\phi_i:i\in I\}$ \emph{has radius at most} $r$ if it is contained in the ball of radius $r$ about $\phi_0$. 
\end{definition}

Note that, contrary to what our terminology might at first suggest, the condition that the forms are $r$-separated is stronger than simply saying that $d(\phi_i,\phi_j)\geq r$ for every $i\ne j$: we need all non-trivial linear combinations to have large support and not just differences between two distinct forms.

We begin with the following simple observation, which we shall use repeatedly.

\begin{remark}\label{union remark} Let $X$ be a finite set, let $q \ge 1$ be a positive integer and let $\epsilon > 0$. Let $\cf_1, \dots, \cf_q, \cg_1, \dots, \cg_q$ be sets of conditions on $X$. If $\cf_i$ internally $\epsilon$-approximates $\cg_i$ for each $i \in \lbrack q \rbrack$ then $\cf_1 \cup \dots \cup \cf_q$ internally $q\epsilon$-approximates $\cg_1 \cup \dots \cup \cg_q$. \end{remark}

Theorem \ref{main approximation theorem} concerns arbitrary sets of mod-$p$ conditions. We shall begin by proving it under a strong additional hypothesis concerning those conditions, and then we shall progressively weaken this hypothesis until we end up proving the result in full generality. At each stage, if we have a proof for families $\cf$ of conditions that satisfy an additional property $P$ and we would like a proof for families $\ce$ of conditions that satisfy a weaker property $Q$, it is sufficient to prove that every family $\ce$ that satisfies $Q$ can be approximated by a family $\cf_1\cup\dots\cup\cf_q$ such that $q$ is bounded and each $\cf_i$ satisfies $P$. Then using the remark above and the result for families that satisfy $P$, we end up with a bounded number of conditions that approximate $\ce$. 

Our first additional hypothesis will be that the set of forms is a symmetric sunflower and that all the subsets of $\F_p$ are the same. Having proved the result under that hypothesis, we shall then weaken the hypothesis by no longer insisting that the set is a sunflower, but keeping the symmetry and the assumption that the subsets of $\F_p$ are the same and adding the condition that the set of forms has bounded radius. Finally, we shall prove the result in full generality. These stages will be the contents of Proposition \ref{Compression of symmetric sunflowers}, Proposition \ref{Compression of symmetric balls}, and Theorem \ref{main approximation theorem} respectively.

Given a linear form $\phi:\F_p^n\to\F_p$ with coefficients $(\lambda_1,\dots,\lambda_n)$, let $D_\phi:\F_p\to\N$ be defined by $D_\phi(y)=|\{i\in[n]:\lambda_i=y\}|$. We shall call $D_\phi$ the \emph{coefficient distribution} of $\phi$. Note that a symmetric family of forms is one where all the forms have the same coefficient distribution.

We begin, then, by proving the result for symmetric sunflowers and identical subsets of $\F_p$. This is a generalization of the example we gave earlier where the forms were $x\mapsto x_1+x_i$, and as with that example it will turn out that as long as there are enough forms, we can $\e$-approximate the set of conditions by a set of the form $(\phi_0,H)$, where $\phi_0$ is the centre of the sunflower.

\begin{proposition} \label{Compression of symmetric sunflowers} Let $p$ be a prime, let $S$ be a subset of $\F_p$ that contains at least two elements, let $\e>0$, let $\{\phi_0+\phi_i:i\in I\}$ be a symmetric sunflower, and let $E\subset\F_p$. Then there exists $H\subset\F_p$ such that the condition $(\phi_0,H)$ $p(1-2^{-(p-1)})^{|I|}$-approximates the set of conditions $(\phi_0+\phi_i,E)$.
\end{proposition}

\begin{proof}

Let $V\subset\F_p$ be the set of all possible values taken by $\phi_i(x)$ when $x\in S^n$, and note that by the symmetry condition $V$ is the same for all $i\in I$. We let $H$ be the set $\{t \in \F_p: V+t\subset E\}$ and we let $T$ be the complement of $H$. If $x\in S^n$ and $\phi_0(x)\in H$, then by the definition of $H$, $\phi_0(x)+\phi_i(x)\in E$. This establishes that the satisfying set for the single condition $(\phi_0,H)$ is contained in the satisfying set for the conditions $(\phi_0+\phi_i,E)$, so it remains to obtain an upper bound for the probability that $\phi_0(x)\in T$ and $\phi_0(x)+\phi_i(x)\in E$ for every $i\in I$.

Let $t\in T$. Then 
\[\P[\phi_0(x)=t\ \wedge\ \forall i\in I\ \phi_0(x)+\phi_i(x)\in E]\leq\P[\forall i\in I\ \phi_i(x)\in E-t]=\prod_{i\in I}\P[\phi_i(x)\in E-t],\]
where the last equality follows from the fact that the $\phi_i$ are disjointly supported and hence the events $\phi_i(x)\in E-t$ are independent (when $x$ is chosen uniformly at random from $S^n$).

Since $V+t$ is not contained in $E$, the probability that $\phi_i(x)\notin E-t$ is non-zero, and therefore by Lemma \ref{Lower bound on non-zero probabilities} it is at least $2^{-(p-1)}$. It follows that the product above is at most $(1-2^{-(p-1)})^{|I|}$. There are at most $p$ possibilities for $t$, so the result follows.
\end{proof}

The proposition above implies that if the set of forms is a symmetric sunflower, then either it is already a family of bounded size, or it can be approximated by a family of size~1.

\begin{proposition}\label{Compression of symmetric balls} For every prime $p$, every non-negative integer $r$ and every $\epsilon>0$ there is a constant $A=A(p,r,\e)$ such that for every pair of subsets $S$ and $E$ of $\F_p$ with $|S|\geq 2$ and every symmetric family $\{\phi_0+\phi_i:i\in I\}$ contained in the ball about $\phi_0$ of radius $r$, the conditions $(\phi_0+\phi_i,E)$ on $S^n$ can be $\e$-approximated by a set of at most $A$ mod-$p$ conditions.
\end{proposition}

\begin{proof} We prove the result by induction on $r$. The proof will be similar to the proof by Erd\H os and Rado of the sunflower lemma (that is, their argument that gives a factorial-type bound for the sunflower conjecture). When $r=0$, the ball of radius $r$ about $\phi_0$ consists of $\phi_0$ only, so the result holds with $A(p,0,\e)=1$. Now let $r \ge 1$ and let us assume that the result holds up to $r-1$. Let $M$ be a maximal subset of $I$ such that the forms $\phi_i$ with $i \in M$ have disjoint supports, and let $Z=\bigcup_{i\in I'}Z(\phi_i)$ be the union of those supports.

If $|M| \ge 2^{p-1} \log (p \epsilon^{-1})$ then by Proposition \ref{Compression of symmetric sunflowers} there exists a subset $H \subset \F_p$ such that the condition $(\phi_0, H)$ $\epsilon$-approximates the conditions $(\phi_0+\phi_i,E)$ with $i\in M$. From the proof of Proposition \ref{Compression of symmetric sunflowers} and the symmetry condition, $(\phi_0, H)$ implies every condition $(\phi_0+\phi_i,E)$ with $i \in I$, so it follows that the condition $(\phi_0,H)$ $\e$-approximates the conditions $(\phi_0+\phi_i,E)$ with $i\in I$.

For each $i$, let the coefficients of $\phi_i$ be $(\lambda_{i1},\dots,\lambda_{in})$. If instead $|M| \le 2^{p-1} \log (p \epsilon^{-1})$ then for each $i \in I$ the support of $\phi_i$ intersects $Z$, so there is some $j \in Z$ and some non-zero $y \in \F_p$ such that $\lambda_{ij} = y$. That is, if we set $I_{j,y}$ to be the set of $i\in I$ such that $\lambda_{ij}=y$, then $I=\bigcup_{y\in\F_p^*}\bigcup_{j\in Z}I_{j,y}$.

Now let us fix $y$ and $j$ and consider the system of forms $\{\phi_0+\phi_i:i\in I_{j,y}\}$. Since each $\lambda_{ij}$ with $i\in I_{j,y}$ is equal to $y$, we can obtain the same system of forms by replacing each $\lambda_{ij}$ with $i\in I_{j,y}$ by $\lambda_{ij}-y$ and replacing $\lambda_{0j}$ by $\lambda_{0j}+y$. Since for each $\phi_i$ with $i\in I_{j,y}$ the support size has been reduced by 1 and the coefficient we have removed is $y$, the result is a symmetric family contained in a ball of radius $r-1$ about the modification of $\phi_0$. 

By the inductive hypothesis, we can therefore $\e/(2^{p-1}pr\log (p \epsilon^{-1}))$-approximate each family $\{\phi_0+\phi_i:i\in I_{j,y}\}$ by a family of at most $A(p,r-1,\e/(2^{p-1}pr\log (p \epsilon^{-1}))$ mod-$p$ conditions. By Remark \ref{union remark} and the fact that $|\F_p^*\times Z|\leq 2^{p-1}pr\log(p\e^{-1})$, it follows that we can $\e$-approximate the family $\{\phi_0+\phi_i:i\in I\}$ by a family of at most 
\[A(p,r,\e)=2^{p-1}pr\log (p \epsilon^{-1})A(p,r-1,\e/(2^{p-1}pr\log (p \epsilon^{-1}))\] 
mod-$p$ conditions.
\end{proof}

A back-of-envelope calculation shows that $A(p,r,\e)$ is at most $(2p)^{2r^2 p}(\log(\e^{-1}))^r$. (We have not attempted to optimize this bound.)

We have now shown that symmetric families of bounded radius can be approximated. It remains to prove that an arbitrary family can be approximated by a union of boundedly many symmetric families of bounded radius. That will be the content of the rest of the argument in this section.

\begin{proof}[Proof of Theorem \ref{main approximation theorem}] Let the set of conditions to be approximated be $\{(\phi_i,E_i):i\in I\}$. We may assume that each $E_i$ is a proper subset of $\F_p$. Let $r = 2p^3 (\log p) \log (2/\e)$.
\medskip

\noindent \textbf{Case 1.} Suppose that there exists a subset $I'$ of $I$ of size at least $p\log(2/\e)$ such that the linear forms $\phi_i$ with $i \in I$ are $r$-separated. Pick such a set that also satisfies the upper bound $|I'|\leq 2p\log(2/\e)$. By Proposition \ref{Equidistribution of k-tuples of linear forms} and our choice of $r$,
\[\P[\forall i \in I', \phi_i=y_i] \le p^{-|I'|}+(1-p^{-2})^{r}\leq 2p^{-|I'|}\]
for each $y\in \F_p^{I'}$.

It follows that 
\[\P[\forall i \in I'\ \phi_i \in E_i] \le (\prod_{i \in I'} |E_i|) (2p^{-|I'|}) \le 2 (1-p^{-1})^{|I'|} \le \e\] 
so the set of conditions $\{(\phi_i,E_i): i \in I\}$ is $\e$-approximated by the condition $(x_1, S^c)$ (or the condition $(\phi,\emptyset)$ for any form $\phi$). 
\medskip

\noindent \textbf{Case 2.} Now suppose that there does not exist such a subset $I'$. Let $M$ be a maximal $r$-separated subset of $I$. Then by hypothesis $M$ has size at most $p \log (2/\e)$, and for every $i \in I$ there exists $a \in \F^M$ such that $\phi_{i} - \sum_{m \in M} a_m \phi_m$ has support size at most $r$. 

As in the proof of Proposition \ref{Compression of symmetric balls}, we now divide the set of conditions into families. This time, given a subset $E\subset\F_p$, an element $a\in\F^M$, and a function $D:\F_p\to\N$ such that $\sum_yD(y)=n$ and $\sum_{y\ne 0}D(y)\leq r$, we let $I_{E,a,D}$ be the set of $i$ such that $E_i=E$, and $\phi_i-\sum_{m\in M}a_m\phi_m$ has coefficient distribution  $D$. There are at most $2^p$ choices of $E$, at most $p^M$ choices of $a$, and at most $r^p$ choices of $D:\F_p\to\N$ such that $\sum_yD(y)=n$ and $\sum_{y\ne 0}D(y)\leq r$, so the number of families $I_{E,a,D}$ we need to approximate is bounded independently of $n$. Also, the forms in each family are a symmetric family of radius at most $r$ and within each family the subsets of $\F_p$ are the same.

Therefore, by Proposition \ref{Compression of symmetric balls} we can $(2^p p^{|M|} r^p)^{-1} \epsilon$-approximate each set of conditions $\{(\phi_i,E_i): i \in I_{E,a,D}\}$ by a set of at most $A(p,r, (2^p p^{|M|} r^p)^{-1} \epsilon)$ conditions. The number of triples $(E,a,D)$ is at most 
\[2^p p^{|M|} r^p \le 2^p p^{p \log (2/\e)} (2p^3 (\log p) \log (2/\e))^p\leq p^{4p\log(2/\e)},\]
so the result now follows by Remark \ref{union remark}, and the number of conditions needed for the approximation is at most
\[p^{4p\log(2/\e)}A(p,2p^3\log p\log(2/\e),p^{-4p\log(2/\e)}\e).\]
After a back-of-envelope calculation we can bound this above by $2^{16p^8\log(2/\e)^2}$. Thus, for fixed $p$ the number of conditions needed for the approximation depends quasipolynomially on $1/\e$.
\end{proof}

\section{More general finite Abelian groups}

In this and the next section we shall prove Theorem \ref{Compression in the groups case for arbitrary alphabets}. This needs extra ideas, but it turns out that those ideas are already needed for the simpler case where $S=G$, so in this section we shall prove that special case, and then in the next section we shall combine the proof method of this section with the ideas about sunflowers from the previous section in order to obtain a proof for general $S$. 

The following definition will be central to the proof.

\begin{definition} Given homomorphisms $\phi_1,\dots,\phi_k:G^b\to H$ and characters $\chi_1,\dots,\chi_k\in\hat H$, the map $x\mapsto\prod_{i=1}^k\chi_i(\phi_i(x))$ defines a character on $G$. Let us call $\phi_1,\dots,\phi_k$ \emph{independent} if the only way for $\prod_i\chi_i\circ\phi_i$ to be the trivial character is if all the $\chi_i$ are the trivial character. 
\end{definition}

\begin{proposition} \label{independence}
Let $\phi_1,\dots,\phi_k$ be independent homomorphisms from $G^n$ to $H$. Then for every $y\in H$, 
\[\P_{x\in G^n}[\forall i\ \phi_i(x)=y_i]=|H|^{-k}.\]
\end{proposition}

\begin{proof}
The probability in question is equal to
\[\E_x\E_{\chi_1,\dots,\chi_k\in\hat H}\prod_{i=1}^k\chi_i(\phi_i(x)-y_i)\]
since if $\phi_i(x)=y_i$ for each $i$, then the inner expectation is equal to 1, and if for some $i$ we have $\phi_i(x)\ne y_i$, then $\E_{\chi_i}\chi_i(\phi_i(x)-y_i)=0$, so the whole expectation is 0.

We can rewrite this expression as
\[\E_{\chi_1,\dots,\chi_k\in\hat H}\prod_{i=1}^k\chi_i(-y_i)\E_x\prod_{i=1}^k\chi_i(\phi_i(x)).\]
If $\chi_1,\dots,\chi_k$ are all equal to the trivial character, then the expression being averaged over is equal to 1. Otherwise, by hypothesis, $\prod_i\chi_i\circ\phi_i$ is not the trivial character on $G^n$, which implies that $\E_x\prod_i\chi_i\circ\phi_i=0$. Therefore, the probability is $|\hat H|^{-k}=|H|^{-k}$. 
\end{proof}

The next definition and lemma are standard, but we recall them here for convenience. 

\begin{definition}
Given an Abelian group $G$, a subgroup $H\subset G$, and a character $\chi\in\hat G$, say that $\chi$ \emph{annihilates} $H$ if $\chi(x)=1$ for every $x\in H$. The \emph{annihilator} of $H$, denoted $H^\perp$, is the subgroup of $\hat G$ that consists of all $\chi$ that annihilate $H$. Similarly, if $\hat H$ is a subgroup of $\hat G$, we say that $x\in G$ \emph{annihilates} $\hat H$ if $\chi(x)=1$ for every $\chi\in\hat H$, and the \emph{annihilator} of $\hat H$ is the subgroup of all $x\in G$ that annihilate $\hat H$. 
\end{definition}

\begin{lemma}\label{isomorphism}
Let $G$ be a finite Abelian group and let $H$ be a subgroup of $G$. Then $H^\perp\cong\widehat{G/H}$.
\end{lemma}

\begin{proof}
Define a map $\iota:H^\perp\to\widehat{G/H}$ by $\iota(\chi)(xH)=\chi(x)$. Note that this is well defined, since if $xy^{-1}\in H$, then $\chi(y)=\chi((xy^{-1})y)=\chi(x)$. It is also a homomorphism, since 
\[\iota(\chi_1\chi_2)(xH)=\chi_1\chi_2(x)=\chi_1(x)\chi_2(x)=\iota(\chi_1)(xH)\iota(\chi_2)(xH)\]
for every $x\in G$. 

Now let $\psi:G/H\to\C$ be a character. Define $\theta(\psi):G\to\C$ by $\theta(\psi)(x)=\psi(xH)$. Then $\theta(\psi)$ is a character on $G$, since
\[\theta(\psi)(xy)=\psi(xyH)=\psi(xHyH)=\psi(xH)\psi(yH)=\theta(\psi)(x)\theta(\psi)(y).\]
It also annihilates $H$, since if $x\in H$, then $\theta(\psi)(x)=\psi(H)=1$. Finally, for any character $\chi$ that annihilates $H$, we have that $\theta\iota\chi(x)=\iota\chi(xH)=\chi(x)$ for every $x$, so $\theta$ is a left inverse for $\iota$. In the other direction, if $\psi:G/H\to\C$ is a character, then 
$\iota\theta\psi(xH)=\theta\psi(x)=\psi(xH)$, so $\theta$ is also a right inverse for $\iota$. 
\end{proof}

\begin{proposition} \label{annihilator}
Let $\phi_1,\dots,\phi_k$ be homomorphisms from $G^n$ to $H$. Let $K\subset\hat H^k$ be the subgroup that consists of all $k$-tuples $(\chi_1,\dots,\chi_k)$ such that $\prod_{i=1}^k\chi_i\circ\phi_i$ is the trivial character on $G^n$. Let $Y\subset H^k$ be the annihilator of $K$. Then 
\[\P[\forall i\ \phi_i(x)=y_i]=\begin{cases}|Y|^{-1}& y\in Y\\ 0&y\notin Y\\ \end{cases}\]
\end{proposition}

\begin{proof}
Again, the probability in question is equal to 
\[\E_{\chi_1,\dots,\chi_k\in\hat H}\prod_{i=1}^k\chi_i(-y_i)\E_x\prod_{i=1}^k\chi_i(\phi_i(x)).\]
If $(\chi_1,\dots,\chi_k)\notin K$, then $\prod_{i=1}^k\chi_i\circ\phi_i$ is not the trivial character, so the expectation over $x$ is 0. Therefore, the probability is equal to
\[(|K|/|\hat H|^k)\E_{(\chi_1,\dots,\chi_k)\in K}\prod_{i=1}^k\chi_i(-y_i).\]
If $(y_1,\dots,y_k)\in Y$, then $\prod_{i=1}^k\chi_i(y_i)=1$ for every $(\chi_1,\dots,\chi_k)\in K$, and therefore we end up with $(|K|/|\hat H|^k)$. But by Lemma \ref{isomorphism} (and Pontryagin duality), $|\hat H|^k/|K|=|Y|$, so this equals $|Y|^{-1}$. 

If $(y_1,\dots,y_k)\notin Y$, then the evaluation map $(\chi_1,\dots,\chi_k)\mapsto\prod_{i=1}^k\chi_i(y_i)$ is a non-trivial character on $K$, and therefore its average is 0. 
\end{proof}

\begin{proposition} \label{relatively independent case}
Let $\phi+\phi_1,\dots,\phi+\phi_k$ be homomorphisms from $G^n$ to $H$, let $E\subset H$, and let $K$ be a subgroup of $\hat H$. Suppose that each $\phi_i$ takes values in a subgroup $H_1$ of $H$ and that the $\phi_i$ are independent when considered as homomorphisms to $H_1$. Then there is a subset $F\subset H$ such that the condition $(\phi,H)$ $\e$-approximates the set of conditions $(\phi+\phi_i,E)$, where $\e=(1-|K|/|\hat H|)^k$.
\end{proposition}

\begin{proof}
Let $F=\{z\in H:z+H_1\subset E\}$. Then if $\phi(x)\in F$, we have for all $i$ that $\phi(x)+\phi_1(x)\in F+H_1\subset E$. Thus, the single condition $(\phi,F)$ implies all the conditions $(\phi+\phi_i,E)$. 

Now let us fix some $z\notin F$ and condition on the event that $\phi(x)=z$. Let $w$ be such that $z+w\notin E$. Then the probability that $\phi(x)+\phi_1(x),\dots,\phi(x)+\phi_k(x)\in E$ given that $\phi(x)=z$ is at most the probability that none of the $\phi_i(x)$ is equal to $w$.

But each $\psi_i$ takes values in $H_1$, and by hypothesis the $\phi_i$ are independent when considered as homomorphisms to $H_1$, so by Proposition \ref{independence} the probability that none of them take the value $w$ is at most $(1-1/|H_1|)^k$.
\end{proof}

\begin{theorem} \label{Abelian groups unrestricted alphabet}
Let $G$ and $H$ be finite Abelian groups and let $\e>0$. Then there exists a constant $C$ such that for every $n$, every set of $H$-conditions on $G^n$ can be $\e$-approximated by at most $C$ $H$-conditions.
\end{theorem}

\begin{proof}
We note first that the number of subsets of $H$ is bounded, so by the usual partitioning argument it is sufficient to prove the result in the case that they are all equal to a single set $E$. 

We prove the result by induction on the size of a subgroup $H_1$ of $H$. The inductive hypothesis is that if the homomorphisms are of the form $\phi+\phi_i$, where each $\phi_i$ takes values in $H_1$, then the result holds. Let us write $C(G,H_1,\e)$ for the number of conditions needed for an $\e$-approximation in this case.

If $H_1$ is the trivial subgroup $\{0\}$, then all the homomorphisms are the same, so all the conditions are the same, and the result is trivial. 

Suppose now that the homomorphisms are of the given form and that we know the result for all proper subgroups of $H_1$. If it is possible to find $k$ of the homomorphisms $\phi_i$ that are independent, where $(1-|H_1|^{-1})^k\leq\e$, then by the first paragraph of the proof of Lemma \ref{relatively independent case} there is a single condition that approximates the entire set of conditions. 

If it is not possible to find more than $k$ of the $\phi_i$ that are independent, then pick a maximal independent set $\{\psi_1,\dots,\psi_{k'}\}$, with $k'\leq k$. Then for every $\phi_i$ there is a combination $(\chi\circ\phi_i)\prod_{j=1}^{k'}(\chi_j\circ\psi_j)$ that equals the trivial character, where all of $\chi_1,\dots,\chi_{k'}$ and $\chi$ are characters on $H_1$ and $\chi$ is not the trivial character. By the usual partitioning argument we can assume that the $(k'+1)$-tuple $(\chi,\chi_1,\dots,\chi_{k'})$ is the same for every $\phi_i$. So that gives us a character $\g:G^n\to\C$ and a character $\chi\in\hat H_1$ such that $\chi\circ\phi_i=\g$ for every $i$.

Let us write $\xi_i=\phi_i-\phi_1$ for each $i$. Then $\phi_i=\phi_1+\xi_i$, and $\chi\circ\xi_i=0$ for every $i$. It follows that each $\xi_i$ takes values in the subgroup $H_2=\{y\in H_1:\chi(y)=0\}$, which is a proper subgroup of $H_1$, as $\chi$ is a non-trivial character on $H_1$. Therefore, by the inductive hypothesis, the conditions $(\phi_i,E)$ can be $\e$-approximated (by at most $C(G,H_2,\e)$ conditions), and we are done.
\end{proof}

\section{General finite Abelian groups with restricted alphabets} \label{restricted alphabet group case}

We now show how to combine the methods of proof of the last two sections to give a proof of Theorem \ref{Compression in the groups case for arbitrary alphabets}, which generalizes Theorem \ref{Abelian groups unrestricted alphabet} to the corresponding statement for restricted alphabets.

Before we start on this, we briefly mention a set-up that appears to be more general than that of Theorem \ref{Compression in the groups case for arbitrary alphabets} but is in fact equivalent to it. Suppose that $H$ is a finite Abelian group, $S$ is a non-empty finite set (so not necessarily a subset of a finite Abelian group), and $\phi_1,\dots,\phi_n$ are functions from $S$ to $H$. We can define a function $\phi:S^n\to H$ by $\phi(x)=\sum_{i=1}^n\phi_i(x_i)$, and if $E\subset H$, then we can consider the condition $(\phi,E)$ on $S^n$, which is the condition $\phi(x)\in E$. This more general looking situation is covered by Theorem \ref{Compression in the groups case for arbitrary alphabets}, because for every $k$ we can find a finite Abelian group $G$ and a subset $S\subset G$ of size $k$ such that every function from $S$ to $H$ can be extended to a homomorphism from $G$ to $H$. Then if we extend functions $\phi_1,\dots,\phi_n$ to homomorphisms $\psi_1,\dots,\psi_n$, we find that the map $\phi$ defined above is the restriction to $S^n$ of the homomorphism $\psi:G^n\to H$ defined by $\psi(x)=\sum_{i=1}^n\psi_i(x)$. Thus, in a certain sense Theorem \ref{Abelian groups unrestricted alphabet} is not really about subsets of $G$ but more about particular kinds of functions from a product $S^n$ to a finite Abelian group $H$. However, the embedding of $S$ into a group $G$ is convenient.
\medskip

Now let us turn to the proof. We begin with some definitions that will generalize those that we used in the case $G=H=\F_p$. Given a homomorphism $\phi:G^n\to H$, we can express it in the form $\phi(x)=\sum_i\phi_i(x_i)$, where each $\phi_i$ is a homomorphism from $G$ to $H$. If $S\subset G$ is non-empty, then we say that the \emph{support of $\phi$ with respect to $S$} (or just the support of $\phi$ if $S$ is clear from the context) is the set of all $i$ such that $\phi_i$ is not constant on $S$. We call the maps $\phi_i|_S$ the \emph{coefficients} of $\phi$ (with respect to $S$), and the \emph{coefficient distribution} of $\phi$ is the function that takes each function $\psi:S\to H$ to the number of $i$ such that $\phi_i|_S=\psi$. The definitions of ``sunflower", ``ball", and ``symmetric sunflower" carry over word for word from our earlier definitions (see Definition \ref{Key structures in the F_p case}) once the words and phrases ``support" and ``coefficient distribution" are interpreted in these generalized senses. We say that homomorphisms $\phi_i: G^n \to H$ defined by $\phi_i(x) = \sum_{j=1}^n \phi_{ij} x_j$ for every $i \in I$ are $r$-separated if whenever $\chi_i \in \hat{H}$ are characters for each $i \in I$ that are not all trivial, the homomorphism \[\prod_{i \in I} \chi_i \circ \phi_i: x \to  \prod_{j=1}^n \prod_{i \in I} \chi_i(\phi_{ij}(x_j))\] has support size at least $r$ (with the support being the number of $j$ such that $\prod_{i \in I} \chi_i(\phi_{ij}(x_j))$ is not constant on $S$).

We begin by proving an analogue of Proposition \ref{Lower bound on non-zero probabilities}, but with a less satisfactory bound.

\begin{proposition} \label{lower bound lemma for groups}
Let $G$ and $H$ be finite Abelian groups, let $S$ be a non-empty subset of $G$, let $n$ be a positive integer, and let $\phi:G^n\to H$ be a homomorphism. Let $x$ be an element of $S^n$, chosen uniformly at random. Then for every $y\in H$, the probability that $\phi(x)=y$ is either 0 or at least $c$, where $c=c(H)$ is a positive constant that depends only on $H$ (and on $S$). 
\end{proposition}

\begin{proof}
Let $Z$ be the support of $\phi$. Then $\phi(x)$ depends only on $x|_Z$, so each value $y$ that is taken with non-zero probability is taken with probability at least $|S|^{-|Z|}$.

We also have that 
\[\P[\phi(x)=y]=\E_x\E_{\chi\in\hat H}\chi(\phi(x)-y)=\E_{\chi\in\hat H}\chi(-y)\E_x\chi(\phi(x))=\E_{\chi\in\hat H}\chi(-y)\prod_i\E_{x_i}\chi(\phi_i(x_i)).\]
Note that $\chi\circ\phi$ is a homomorphism and its support is the number of $i$ such that $\chi\circ\phi_i$ is not constant on $S$. For each such $i$, $|\E_{x_i}\chi(\phi_i(x_i))|$ is at most $|1+\exp(2\pi i/k)|/2$, where $k$ is the order of $\chi$ in $\hat H$. It is therefore at most $1-e(H)^{-2}$, where $e(H)$ is the exponent of~$H$. 

It follows that if for each non-trivial character $\chi$ the support size of $\chi\circ\phi$ is at least $r$, then for every $y$ the difference between $\P[\phi(x)=y]$ and $|H|^{-1}$ is at most $\exp(-re(H)^{-2})$. So if $r$ is at least $e(H)^2\log(2|H|)$, then the probability that $\phi(x)=y$ is at least $|H|^{-1}/2$ for all $y$. 

Suppose that there is some non-trivial character $\chi$ such that $\chi\circ\phi$ has support size less than $r$. Without loss of generality the support of $\chi$ is $\{1,2,\dots,m\}$ for some $m\leq r$. Let us write $x=x^1+x^2$ where $x^1=(x_1,\dots,x_m,0,\dots,0)$ and $x^2=(0,\dots,0,x_{m+1},\dots,x_n)$. Then $\chi(\phi(x))$ depends only on $x^1$, or to put it another way, $\phi(x^2)$ belongs to a certain coset of $H_1=\ker\chi$. Thus, for each coset $z+H_1$, either there is no $x$ with $\phi(x)\in z+H_1$ or the probability that $\phi(x)\in z+H_1$ is at least $|S|^{-r}$. Suppose now that there exists $x$ with $\phi(x)\in z+H_1$ and pick $x^1$ such that $\phi(x^1+x^2)\in z+H_1$ for every $x^2$. Then for each $y\in z+H_1$ such that there exists $x^2$ with $\phi(x^1+x^2)=y$, the probability that $\phi(x^1+x^2)=y$ is, by induction on $|H|$, at least $c(H_1)$. It follows that we may take $c(H)$ to be $|S|^{-e(H)^2\log(2|H|)}c(H_1)$. Since $|H_1|\leq|H|/2$ and $e(H_1)\leq e(H)$, an easy induction shows that we may therefore take $c(H)$ to be at least $|S|^{-e(H)^2(\log(2|H|))^2}$.
\end{proof}

The next lemma and its proof combine Proposition \ref{Compression of symmetric sunflowers} with Proposition \ref{relatively independent case}. We stress that, consistent with the definition that we have given, maps $\theta_i$ consisting merely of the same map repeated several times form a sunflower.

\begin{lemma} \label{Compression of symmetric sunflower/separated pairs in the groups case}
Let $G$ and $H$ be finite Abelian groups, let $S$ be a non-empty subset of $G$, and let $H_1$ be a subgroup of $H$. Let $\Phi=\{\phi_1,\dots,\phi_m\}$ be a set of $m$ homomorphisms from $G^n$ to $H$ and suppose that we can write $\phi_i=\theta_i+\psi_i$ in such a way that the maps $\theta_i$ form a symmetric sunflower contained in a ball of radius $r$, that the maps $\psi_i$ take values in $H_1$ and are $(rm+r')$-separated. Then for every proper subset $E\subset H$ the conditions $(\phi_i,E)$, regarded as conditions on $S^n$, can be $\e$-approximated by a single condition, where 
\[\e=|H|\Big(\exp\big(-m|S|^{-e(H)^2(\log(2|H|))^2}|H_1|\big)+|H_1|^m\exp\big(-r'e(H_1)^{-2}\big)\Big).\]
\end{lemma}

\begin{proof}
Let $\theta_i=\g+\zeta_i$, where the $\zeta_i$ have disjoint supports of size at most $r$ and the same coefficient distribution. For each $x\in H$, let $E_x=\{y\in H_1:x+y\in E\}$. We now let $F\subset H$ be the set of $t\in H$ such that $E_{t+\zeta_i(x)}=H_1$ for every $x\in S^n$. Note that since the $\zeta_i$ have the same coefficient distribution, their images are the same, so $H$ does not depend on $i$.

We shall now show that the single condition $(\g,F)$ $\e$-approximates the set of conditions $(\phi_i,E)$. First, observe that if $\g(x)\in F$, then for every $i$, 
\[E_{\theta_i(x)}=E_{\g(x)+\zeta_i(x)}=H_1,\]
and therefore, since $\psi_i$ takes values in $H_1$, $\phi_i(x)=\theta_i(x)+\psi_i(x)\in E$. Thus the condition $(\g,F)$ implies all the other conditions. 

It remains to prove that the set of $x\in S^n$ such that $\g(x)\notin F$ and $\phi_i(x)\in E$ for every $i$ has density at most $\e$. For this purpose let us fix $y\notin F$ and obtain an upper bound for the probability that $\g(x)=y$ and $\phi_i(x)\in E$ for every $i$. Note that if $\g(x)=y$, then $\phi_i(x)\in E$ if and only if $\psi_i(x)\in E_{y+\zeta_i(x)}$ for every $i$. Since $y\notin F$, there exists $x\in S^n$ such that $E_{y+\zeta_i(x)}\ne H_1$. That is, there exist $x\in S^n$ and $w\in H$ such that $\zeta_i(x)=w$ and $E_{y+w}\ne H_1$. By Proposition \ref{lower bound lemma for groups} the probability that $\zeta_i(x)=w$, when $x$ is chosen uniformly from $S^n$, is at least $|S|^{-e(H)^2(\log(2|H|))^2}$, from which it follows that the probability that $E_{y+\zeta_i(x)}\ne H_1$ is at least $|S|^{-e(H)^2(\log(2|H|))^2}$. Let $m(x)$ be the number of $i$ such that $E_{y+\zeta_i(x)}\ne H_1$.

Now let $Z$ be the union of the supports of the maps $\zeta_i$. Then $Z$ has cardinality at most $rm$. Let us write $W$ for $Z^c$ and $\pi_Z$ and $\pi_W$ for the coordinate projections onto $Z$ and $W$. Then since the maps $\psi_i$ are $(rm+r')$-separated, it follows that the maps $\psi_i\circ\pi_W$ are $r'$-separated. Applying the beginning of the proof of Proposition \ref{lower bound lemma for groups} to the map $\psi=(\psi_1\circ\pi_W,\dots,\psi_m\circ\pi_W)$, we find that for each $(z_1,\dots,z_m)\in H_1^m$ the probability that $\psi_i(\pi_Wx)=z_i$ for $i=1,2,\dots,m$ differs from $|H_1|^{-m}$ by at most $\exp(-r'e(H_1)^{-2})$, since $e(H^m)=e(H)$. Therefore, by conditioning on the restriction of $x$ to $Z$ and applying the law of total probability we find that the probability that $\psi_i(x)\in E_{y+\zeta_i(x)}$ for every $i$ is at most $(1-|H_1|^{-1})^{m(x)}+|H_1|^m\exp(-r'e(H_1)^{-2})$.

We now drop the conditioning. Let $q=|S|^{-e(H)^2(\log(2|H|))^2}$. Then from the above calculations it follows that the probability that $\psi_i(x)\in E_{y+\zeta_i(x)}$ for every $i$ is at most
\begin{align*}\sum_{j=0}^m\binom mj(1-q)^{m-j}&q^j(1-|H_1|^{-1})^{j}+|H_1|^m\exp(-r'e(H_1)^{-2})\\
&=(1-q|H_1|)^m+|H_1|^m\exp(-r'e(H_1)^{-2})\\
&\leq\exp(-mq|H_1|)+|H_1|^m\exp(-r'e(H_1)^{-2}),\\
\end{align*}
where the equality follows from the binomial theorem.

Since there are at most $|H|$ possibilities for $y$, we obtain the bound stated. 
\end{proof}

\begin{remark} \label{remark about bounds} Note that if we wish to obtain an $\e$-approximation, then it is sufficient if $m\geq\log(2|H|/\e)|S|^{e(H)^2(\log(2|H|))^2}|H_1|$ and $r'\geq\log(2|H||H_1|^m/\e)e(H_1)^2$. Note also that the proof that the condition $(\g,F)$ implies all the other conditions used only the fact that the $\theta_i$ have the same coefficient distribution, so if $t=\lceil\log(2|H|/\e)|S|^{e(H)^2(\log(2|H|))^2}|H_1|\rceil$, then to obtain an $\e$-approximation it is sufficient if the $\theta_i$ belong to a symmetric ball of radius $r$ and we can find $t$ of them that form a sunflower. Taking the contrapositive, if the $\theta_i$ belong to a symmetric ball and $(\g,F)$ does \emph{not} $\e$-approximate the set of conditions $(\phi_i,E)$, then it is not possible to find a set $I$ of size $t$ such that the conditions $\theta_i$ with $i\in I$ form a sunflower and the conditions $\psi_i$ with $i\in I$ are $\log(2|H||H_1|^t/\e)e(H_1)^2$-separated.
\end{remark}

We are now ready to prove the main result of the section.

\begin{proof}[Proof of Theorem \ref{Compression in the groups case for arbitrary alphabets}] We prove by induction on the size of $|H_1|$ and on the radius $r_{H_1}$ that for every $\e>0$ the theorem holds (with bounds depending only on $\e$, $G$, $S$, $H$, $H_1$, $r_{H_1}$) for every set of conditions $(\theta_i+\psi_i,E)$ such that each $\theta_i$ is a map from $G$ to $H$, each $\psi_i$ is a map from $G$ to $H_1$, and the $\theta_i$ are contained in a symmetric ball of radius at most $r_{H_1}$.

We again write $\theta_i=\g+\zeta_i$, where the $\zeta_i$ have disjoint supports of size at most $r_{H_1}$ and the same coefficient distribution.

We first describe the inductive step in the case where $r_{H_1} > 0$ and $H_1$ is not the trivial subgroup. The cases where  $H_1 = \{e\}$ or $r_{H_1} = 0$ can be dealt with by a degenerate version of the same argument, as we will explain at the end of the proof.

Suppose that the result holds for all $\e>0$ and for all pairs $(H,r)$ such that $H$ is a strict subgroup of $H_1$, or such that $H=H_1$ and $r < r_{H_1}$. Let $\e>0$. Then by Lemma \ref{Compression of symmetric sunflower/separated pairs in the groups case} (and Remark \ref{remark about bounds}) we are done if we can find $t_{H_1}=\lceil\log(2|H|/\e)|S|^{e(H)^2(\log(2|H|))^2}|H_1|\rceil$ of the maps such that the maps $\theta_i$ form a sunflower with centre $\g$ and the maps $\psi_i$ are $(r_{H_1}t_{H_1}+r_{H_1}')$-separated for $r_{H_1}' = \lceil \log(2|H||H_1|^m/\e)e(H_1)^2 \rceil$. If we cannot find such a set of maps which is that large, then let $M$ be a maximal set of such maps.

%such that the maps $\psi_m$ with $m\in M$ are $(r_{H_1}t_{H_1}+r_{H_1}')$-separated.

Then for each $i$, the support of the map $\zeta_i$ contains a support element inside $Z$, or there exists a non-trivial character $\chi\in\hat H_1$ such that $\chi\circ\psi_i$ has support distance at most $r_{H_1}t_{H_1}+r_{H_1}'$ from a product of the form $\prod_{m\in M}\chi_m\circ\psi_m$, meaning that the set of $j$ such that \[x_j \to \chi_i(\psi_{ij}(x_j)) (\prod_{m \in M} \chi_i(\psi_{mj}(x_j)))^{-1}\] is constant on $S$ has size at most $r_{H_1}t_{H_1}+r_{H_1}'$. By Remark \ref{union remark} we may assume that the first situation occurs for all $i$, or that the second situation occurs for all $i$.

Assume that we are in the first situation. Let $d$ be the minimal size of a generating set of $G$. The number of homomorphisms from $G$ to $H$ is at most $|H|^d$, so the number of pairs $(\xi,z)$ where $\xi$ is a non-zero homomorphism from $G$ to $H$ and $z$ is an element of $Z$ is at most $|H|^d|M|$. We now partition the set of indices $i$ into sets $I_{(\xi,z)}$ in such a way that for every pair $(\xi,z)$, the coefficient of the map $\xi_i$ at $z$ is equal to $\xi$, in our generalized sense of the word ``coefficient".

We claim now that each set $\{\theta_i:i\in I_{(\xi,z)}\}$ is contained in a ball of radius at most $r_{H_1}-1$. To see this, pick $\xi$ and $z$ as above, and let $\g':G^n\to H$ be the homomorphism that takes $x$ to $\g + \xi(x_q)$. Then each map $\theta_i-\g'$ with $i\in I_{(\xi,z)}$ is a homomorphism of support size at most $r_{H_1}-1$, so the homomorphisms $\theta_i= \g' + (\theta_i-\g')$ live in a ball of radius at most $r_{H_1}-1$ with centre $\g'$. By Remark  \ref{remark about bounds} it suffices to approximate each set $\{(\theta_i + \psi_i, E_i): i \in I_{(\xi,z)}\}$, which in turn we can do by the inductive hypothesis.

Assume that we are instead in the second situation. There are at most $|H_1|^{t_{H_1}+1}$ choices for $\chi$ and the $\chi_m$ with $m \in M$, and therefore by Remark \ref{union remark} we may assume that these choices are the same for each $j$. We write $\Phi$ for the corresponding set of maps $\phi_i = \theta_i + \psi_i$. We now have that there exist a non-trivial character $\chi\in\hat H_1$ and a character $\chi_m$ for each $m \in M$ such that for every $\phi=\theta+\psi\in\Phi$ the composition $\chi\circ\psi$ is within support distance $r_{H_1}t_{H_1}+r_{H_1}'$ of $\prod_{m \in M}\chi_m \circ\psi_m$.

Let us now choose an arbitrary $\phi_0=\theta_0+\psi_0\in\Phi$. Then for each $\phi=\theta+\psi\in\Phi$, the homomorphism $\chi\circ(\psi-\psi_0)$ has support size at most $2(r_{H_1}t_{H_1}+r_{H_1}')$, and therefore we can write $\psi-\psi_0$ as a sum of a map $\theta'$ of support size at most $2(r_{H_1}t_{H_1}+r_{H_1}')$ and a map $\psi'$ that takes values in the subgroup $H_2=\ker\chi$ of $H_1$. Since $\chi$ is not the trivial character, $H_2$ is not the whole of $H_1$. 

To summarize, we can write each $\phi_i=\theta_i+\psi_i$ as $\theta_i+\theta_i'+\psi_0+\psi_i'$, where the $\theta_i$ form a symmetric sunflower of radius at most $r_{H_1}$, the $\theta_i'$ have support size at most $2(r_{H_1}t_{H_1}+r_{H_1}')$, $\psi_0$ is fixed, and the $\psi_i'$ take values in $H_2$. Note that the maps $\theta_i+\theta_i'+\psi_0$ are contained in a ball of radius at most $r_{H_2}=r_{H_1}(2t_{H_1}+1)+2r_{H_1}'$. 

Let us write $\theta_i+\theta_i'+\psi_0=\gamma+\eta_i$, where $\g$ is the centre of the ball and $\eta_i$ is a map of support size at most $r_{H_2}$. There are at most $|H|^d$ homomorphisms from $G$ to $H$ (where again $d$ is the minimal size of a generating set of $G$), and therefore at most $|H|^{dr_{H_2}}$ possible coefficient distributions for the $\eta_i$. So by Remark \ref{union remark} we may assume that the $\eta_i$ have the same coefficient distribution, and then conclude by the inductive hypothesis. The inductive step is now complete.

In the case where $H_1 = \{e\}$ and $r_{H_1}=0$, which is the base case of the induction, the result is immediate. If only the first or only the second of these equalities hold, then the inductive step is as described above, but with only the second situation or only the first situation, respectively, rather than both of them.

%By Lemma \ref{sunflower type lemma}, either the multiset of homomorphisms $\gamma+\eta_i$ contains a symmetric sunflower of size $t_{H_2}$ or it has size at most $r_{H_2}!t_{H_2}^{r_{H_2}+1}|H|^{(r_{H_2}^2+3r_{H_2})d/2}$. In the second case we have only a bounded number of conditions, so we are done. In the first case, we are done by the inductive hypothesis and the second part of Remark \ref{remark about bounds}.

It remains to remark that we assumed that there was a set $E$ such that $E_i=E$ for all the conditions $(\phi_i,E_i)$ in the collection. As usual, this is permissible by Remark \ref{union remark}.
\end{proof}

\section{Polynomial conditions with restricted alphabets} \label{Section: Compression of polynomial conditions in restricted subsets of F_p^k to a single condition}

The approximation results we have proved so far have concerned linear conditions. The aim of this section is to prove an approximation result for polynomial conditions. Let $V$ be a $k$-dimensional vector space over $\F_p$. We shall call a map $\phi:\F_p^n\to V$ a \emph{polynomial map of degree at most} $d$ if the map $a^*\phi:\F_p^n\to\F_p$ defined by $x\mapsto a^*\phi(x)$ is a polynomial of degree at most $d$ for each $a^*\in V^*$. If $V=\F_p^k$, then we can regard a polynomial map $\phi:\F_p^n\to\F_p^k$ as a $k$-tuple $(\phi_1,\dots,\phi_k)$ of polynomials on $\F_p^n$, as we did in the introduction, in which case the degree is simply the maximum degree of the $\phi_i$, but for most of the argument it will be convenient to think of polynomials in this more basis-free way. (However, we do not replace $\F_p^n$ by an arbitrary $n$-dimensional vector space, since we are considering sets of the form $S^n$, which are tied to the standard basis of $\F_p^n$.)

Thus, we now define a \emph{$k$-dimensional condition with degree at most $d$} to be a pair $(\phi,E)$, where $\phi:\F_p^n\to V$ for some $k$-dimensional vector space $V$ over $\F_p$, $\phi$ is a polynomial map of degree at most $d$, and $E\subset V$. We shall prove Theorem \ref{Compression of polynomial conditions to a single condition}, which stated that for every $\e>0$, every prime $p$, every subset $S\subset\F_p$, every positive integer $k$ and every positive integer $d$ there exists $K$ such that every set of $k$-dimensional conditions with degree at most $d$ on $S^n$ can be $\e$-approximated by a polynomial condition of degree at most $d$ and dimension at most $K$. 

As commented in the introduction, this result is a little different from our earlier approximation results, since we approximate by a single condition, but to pay for that we allow the dimension of that condition to be larger than the dimensions of the conditions we started with. It is not clear whether the natural direct generalization of our linear results is true -- there does not seem to be an obvious analogue of the sunflower structure that played an important role in the proofs.

Given a $k$-dimensional space $V$ and a polynomial map $\phi:\F_p^n\to V$ of degree at most $d$, for each $0\leq i\leq d$ let $W_i^*\subset V^*$ be the space of all $a^*\in V^*$ such that $a^*\phi$ has degree at most $i$. We can take a basis of $W_0^*$, extend it to a basis of $W_1^*$, and so on. Let $V_i^*$ be the subspace of $V^*$ spanned by the basis vectors that belong to $W_i^*\setminus W_{i-1}^*$. (Since the basis is not uniquely determined, the subspaces $V_i^*$ are not uniquely determined either, but this will not matter to us.) Let $k_i=\dim V_i^*$ and note that $V^*=V_0^*\oplus\dots\oplus V_d^*$, so $k_0+k_1+\dots+k_d=k$. We shall call the sequence $(k_0,k_1,\dots,k_d)$ the \emph{degree profile} of the polynomial $\phi$. Let $V=V_0\oplus\dots\oplus V_d$ be the corresponding decomposition of $V$ -- that is, we take the dual basis and the subspaces spanned by the corresponding subsets. Let $P_i$ be the projection from $V$ to $V_i$ (that is, the coordinate projection with respect to this basis) and note that if $a^*$ is a non-zero element of $V_i^*$, then $a^*P_i\phi$ has degree exactly $i$.

We shall prove Theorem \ref{Compression of polynomial conditions to a single condition} by induction on the degree profile in the colex ordering. That is, we shall prove the result for polynomial conditions where the polynomials have a given degree profile $(k_0,k_1,\dots,k_d)$ assuming the result for polynomial conditions with degree profile $(k_0',k_1',\dots,k_d')$ such that there exists $i$ for which $k_i'<k_i$ and $k_j'=k_j$ for every $j>i$.

The proof will appear somewhat simpler than several of the proofs in earlier sections. This is partly because allowing the dimension to increase gives us considerable extra flexibility, but the simplicity is also partly illusory, as the proof relies on the main result of \cite{G and K equidistribution}, which in turn relies on one of the main results of \cite{K}. 

The result from \cite{G and K equidistribution} concerns the equidistribution of high rank polynomials on sets of the form $S^n$, where $S\subset\F_p$ is non-empty. When $S=\F_p$ it is a result of Green and Tao \cite[Theorem 1.7]{Green and Tao}. Before stating it, let us define a notion of rank for polynomials on restricted alphabets. The definition we give here is less granular than the one we gave in the paper \cite{G and K equidistribution}, but it suffices for our purposes and will be slightly more convenient to use.

\begin{definition} \label{rank of a polynomial over F_p} Let $p$ be a prime, let $d$ be a nonnegative integer, and let $P: \F_p^n \rightarrow \F_p$ be a polynomial with degree at most $d$. We define the \emph{degree-$d$ rank} $\rk_d P$ of $P$ as follows.

\begin{enumerate}
\item If $d=0$ then $\rk_d P = 0$.
\item If $d=1$ then $\rk_d P$ is the number of indices $i \in \lbrack n \rbrack$ such that $a_{i} \neq 0$ when we represent $P$ in the form $P(x) = c + \sum_{i=1}^n a_i x_i$, where $c, a_1, \dots, a_n \in \F_p$ -- that is, the size of the support of the linear form that differs from $P$ by an additive constant.
\item If $d \ge 2$ then $\rk_d P$ is the smallest nonnegative integer $r$ such that there exists a function $F: \F_p^r \rightarrow \F_p$ and polynomials $P_1, \dots, P_r: \F_p^n \rightarrow \F_p$ each with degree strictly smaller than $\deg P$ such that $P = F(P_1, \dots, P_r)$. 
\end{enumerate} 
If $S$ is a non-empty subset of $\F_p$, then the \emph{degree-$d$ rank of $P$ with respect to $S$}, denoted by $\rk_{d,S} P$, is $\min_{P_0} \rk_d (P-P_0)$, where the minimum is taken over all polynomials $P_0: \F_p^n \rightarrow \F_p$ such that $P_0(S^n) = \{0\}$. Equivalently, it is the minimum degree-$d$ rank of any polynomial $Q$ that agrees with $P$ on $S^n$.
\end{definition}

The next result is very similar to Theorem 1.4 of \cite{G and K equidistribution} and is an immediate consequence of it. It shows that the values of a high-rank polynomial are approximately equidistributed on sets of the form $S^n$.

\begin{theorem} \label{Equidistribution of polynomials} For every prime $p$, every non-empty subset $S\subset\F_p$, every positive integer $d<p$, and every $\e>0$ there exists $\tau$ such that if $P:\F_p^n\to\F_p$ is any polynomial of degree $d$ with $\rk_{d,S}(P)\geq\tau$, then $|\E_{x \in S^n} \omega_p^{P(x)}| \le \epsilon$. \end{theorem}

We shall also simplify the argument in a superficial way by being less explicit about how various bounds depend on each other. Given a property $\kappa$ of polynomial conditions such that every condition that satisfies $\kappa$ has degree at most $d$, let us say that conditions that satisfy $\kappa$ are \emph{boundedly approximable on $S^n$}, or simply that they are \emph{approximable}, if for every $\e>0$ there exists $K=K(\e,\kappa,S)$ such that every set of polynomial conditions that satisfy $\kappa$ can be approximated by a polynomial condition of degree at most $d$ and dimension at most $K$. We shall frequently make implicit use of the following lemma, which is very similar to Remark \ref{union remark}.

\begin{lemma} \label{union lemma}

Let $\kappa$ be a property of polynomial conditions, and let $\mathcal{P}$ be a family of properties of polynomial conditions. Suppose that there exists a constant $M=M(\kappa,S)$ such that every set of conditions $\mathcal C$ that satisfy $\kappa$ is a union of at most $M$ sets of conditions each satisfying a property of $\mathcal{P}$. If for every $\lambda \in \mathcal{P}$, conditions that satisfy $\lambda$ are approximable on $S^n$, then conditions that satisfy $\kappa$ are approximable on $S^n$.

\end{lemma}

\begin{proof}
Let $\e>0$ and let $\mathcal C$ be a set of conditions that satisfy $\kappa$ and write $\mathcal C$ as $\cc_1\cup\dots\cup\cc_m$, where $m\leq M$ and the conditions in each $\cc_i$ satisfy $\lambda_i$. For each set $\cc_i$, let $(\phi_i,E_i)$ be a condition that $(\e/m)$-approximates $\cc_i$ (or more precisely, the set $\{(\phi_i,E_i)\}$ approximates $\cc_i$), where $E_i\subset\F_p^{k_i}$ for some $k_i\leq K$. Then let $k=k_1+\dots+k_m$, let $\phi:S^n\to\F_p^k$ be defined by $\phi(x)=(\phi_1(x),\dots,\phi_m(x))$, and let $E=E_1\times\dots\times E_m$. Then $x$ satisfies $(\phi,E)$ if and only if it satisfies $(\phi_i,E_i)$ for each $i$. Also, if $x$ satisfies $\cc$ but not $(\phi,E)$, then there is some $i$ such that it fails to satisfy $(\phi_i,E_i)$ but satisfies $\cc_i$. For each $i$, the density of such $x$ is at most $\e/m$, so we are done by a union bound. \end{proof}

If we have a set of conditions that we wish to approximate, this lemma allows us to partition it into a bounded number of pieces and approximate each piece. We have done this several times already in the paper, but from now on we shall not bother to give a careful justification each time we do so.

We shall also use the following more elaborate version of the principle, which relies heavily on the fact that we allow ourselves to increase the dimensions of conditions as we approximate. If $\kappa$ is a property of polynomials, we shall say that a set of conditions satisfies $\kappa$ if the corresponding polynomials satisfy $\kappa$.

\begin{lemma} \label{ball lemma}
Let $p$ be a prime, let $S$ be a non-empty subset of $\F_p$, let $k$ and $d$ be positive integers, let $V$ be a $k$-dimensional vector space over $\F_p$, let $\rho:\F_p^n\to V$ be a polynomial map of degree at most $d$, and let $\kappa$ and $\lambda$ be properties of polynomials that imply that they have degree at most $d$. Suppose that $k$ is bounded above by a function of $p,S$ and $d$. Suppose also that for every $y\in V$ and every polynomial $\phi$ that satisfies $\kappa$ there is a polynomial $\psi$ that satisfies $\lambda$ such that $\phi$ and $\psi$ agree on the set $\{x\in\F_p^n:\rho(x)=y\}$. Finally, suppose that polynomial conditions that satisfy $\lambda$ are approximable on $S^n$. Then polynomial conditions that satisfy $\kappa$ are approximable on $S^n$. 
\end{lemma}

\begin{proof}
Let $\cc$ be a set of conditions that satisfy $\kappa$. For each $y\in V$ and each condition $(\phi,E)\in\cc$, where $E$ is a subset of a vector space $W$, let 
\[E_y=\{(v,w)\in V\times W: v=y\implies w\in E\},\]
so that $((\rho,\phi),E_y)$ is the condition $\rho(x)=y\implies\phi(x)\in E$.

For each $y\in V$, let $\cc_y$ be the set of conditions $\{((\rho,\phi),E_y):(\phi,E)\in\cc\}$. If $x$ satisfies $(\phi,E)$ then it clearly satisfies $((\rho,\phi),E_y)$ for every $y\in V$, and the converse is true as well, since if we let $y=\rho(x)$, then the fact that $x$ satisfies $((\rho,\phi),E_y)$ guarantees that $\phi(x)\in E$, as noted above. Therefore, the satisfying sets of $\cc$ and $\bigcup_{y\in V}\cc_y$ are the same. 

By Lemma \ref{union lemma} it is therefore sufficient to prove that each set $\cc_y$ of conditions is approximable. Let us fix $y\in V$, and for each $(\phi,E)\in\cc$ let $\tilde\phi$ be a polynomial that satisfies $\lambda$ and that agrees with $\phi$ on $\{x\in\F_p^n:\rho(x)=y\}$. Let $\cd_y$ be the set of conditions $(\tilde\phi,E):(\phi,E)\in\cc\}$ and let $\tilde\cc_y$ be the set of conditions $\{((\rho,\tilde\phi),E_y):(\phi,E)\in\cc_y\}$. Note that the satisfying set of $\tilde\cc_y$ is the same as that of $\cc_y$.

Since the polynomials $\tilde\phi$ satisfy $\lambda$, and polynomial conditions that satisfy $\lambda$ are approximable, $\cd_y$ is approximable by a condition $(\theta,F)$. It follows that $\tilde\cc_y$ is approximable by the condition $((\rho,\theta),F_y)$. \end{proof}

We are now ready to begin the proof in earnest. The basic strategy is the same as it has been with our earlier proofs: either the conditions we wish to approximate contain a large set of ``sufficiently independent" conditions, in which case the satisfying set has small density, or there is a small set of conditions that in some suitable sense ``approximately generates" all the other conditions, in which case we can use Lemma \ref{ball lemma} to reduce to a simpler case that is covered by an inductive hypothesis.

Let us begin with a standard condition for determining whether a function taking values in a $k$-dimensional vector space over $\F_p$ is approximately equidistributed. 

\begin{lemma} \label{general equidistribution lemma}
Let $X$ be a finite set, let $p$ be a prime, let $k$ be a positive integer, let $V$ be a $k$-dimensional vector space over $\F_p$, and let $\phi:X\to V$. Then for every $y\in V$,
\[\P[\phi(x)=y]-p^{-k}|\leq p^{-k}\sum_{{a^*\in V^*}\atop{a^*\ne 0}}|\E_x\omega_p^{a^*(\phi(x))}|.\]
\end{lemma}

\begin{proof}
Note first that $\P[\phi(x)=y]=\E_x \mathbbm 1_y(\phi(x))$, where we write $\mathbbm 1_y$ for the characteristic function of the singleton set $\{y\}$. By the Fourier inversion formula,
\[\E_x\mathbbm 1_y(\phi(x))=\E_x\sum_{a^*\in V^*}\widehat{\mathbbm 1_y}(a^*)\omega_p^{a^*(\phi(x))}=p^{-k}\sum_{a^*\in V^*}\E_x\omega_p^{a^*(\phi(x)-y)}.\]
The result follows. 
\end{proof}

The above lemma immediately implies a version for several functions, but it is convenient to state it separately.

\begin{corollary} \label{equidistribution lemma for many functions}
Let $X$ be a finite set, let $p$ be a prime, let $r$ and $k_1,\dots,k_r$ be positive integers, let $k=\sum_ik_i$, for each $i\leq r$ let $V_i$ be a vector space over $\F_p$ of dimension $k_i$, and let $\phi_i:X\to V_i$. Then for every $(y_1,\dots,y_r)\in V_1\times\dots\times V_r$,
\[|\P[\forall i\ \phi_i(x)=y_i]-p^{-r}|\leq p^{-r}\sum_{{a^*\in V_1^*\times\dots\times V_r^*}\atop{a^*\ne 0}}|\E_x\omega_p^{\sum_ia_i^*(\phi_i(x))}|.\]
\end{corollary}

\begin{proof}
We can apply the lemma to the map $\phi:X\to V_1\times\dots\times V_r$ defined by $x\mapsto(\phi_1(x),\dots,\phi_r(x))$.
\end{proof}

A simple further consequence of this corollary and Theorem \ref{Equidistribution of polynomials} is the following result.

\begin{corollary} \label{bias implies low rank}
For every $\e>0$, every pair of positive integers $d$ and $k$, every prime $p$ and every non-empty subset $S\subset\F_p$ there exist $R,\tau$ such that for every $r\geq R$, every sequence $V_1,\dots,V_r$ of vector spaces over $\F_p$ with $\dim V_i\leq k$ for each $i$, every sequence of polynomial maps $\phi_1,\dots,\phi_r$ of degree at most $d$ with $\phi_i:\F_p^n\to V_i$, and every sequence $E_1,\dots,E_r$ with $E_i$ a proper subset of $V_i$, if the satisfying set of the conditions $(\phi_i,E_i)$ has density greater than $\e$, then there exist $a_1^*,\dots,a_r^*$ with $a_i^*\in V_i^*$, not all zero, such that $\sum_ia_i^*\phi_i$ is a polynomial of degree $d'\leq d$ with $\rk_{d',S}(\sum_ia_i^*\phi_i)\leq\tau$.
\end{corollary}

\begin{proof}
If $(1-p^{-k})^R\leq\e/2$, then the density of the set $E_1\times\dots\times E_r$ in $V_1\times\dots\times V_r$ is at most $\e/2$, so by Corollary \ref{equidistribution lemma for many functions} with $X=S^n$ we find that
\[\P_{x\in S^n}[(\phi_1(x),\dots,\phi_r(x))\in E_1\times\dots\times E_r]\leq\e/2+\sum_{{a^*\in V_1^*\times\dots\times V_r^*}\atop{a^*\ne 0}}|\E_x\omega_p^{\sum_ia_i^*(\phi_i(x))}|.\]
Therefore, if the satisfying set of the conditions has density at least $\e$, then there exists a non-zero $a^*\in V_1^*\times\dots\times V_r^*$ such that
\[|\E_x\omega_p^{\sum_ia_i^*(\phi_i(x))}|\geq\e/2p^{rk}.\]

Let $\psi=\sum_ia_i^*\phi_i$ and let $d'$ be the degree of $\psi$. By Theorem \ref{Equidistribution of polynomials} there is a constant $\tau$ depending only on $p,S,\e,r$ and $k$ such that $\rk_{d',S}(\psi)\leq\tau$, which proves the result.
\end{proof}

We now prove an easy result that will turn out to be the base case of an inductive argument. It also has a very similar structure to that of the main proof. Let us define the \emph{support} $Z(\phi)$ of a linear map $\phi:\F_p^n\to V$ to be the set of coordinates on which $\phi$ depends. That is, if $e_1,\dots,e_n$ is the standard basis, then $Z(\phi)=\{i:\phi(e_i)\ne 0\}$. If $\phi$ is an affine map, then we define $Z(\phi)$ to be the support of its linearization, and we also call it the support.

\begin{lemma} \label{bounded support linear case}
Let $p$ be a prime, let $s$ and $k$ be positive integers and let $S$ be a non-empty subset of $\F_p$. Then conditions of the form $(\phi,E)$ where $\phi$ is an affine map from $\F_p^n$ to some vector space $V$ of dimension at most $k$, $E\subset V$, and $\phi$ has support size at most $s$, are approximable on $S^n$. 
\end{lemma}

\begin{proof}
We prove the result by induction on $s$, observing that it is trivial when $s=0$. We now assume that the result holds for $s-1$ and prove it for $s$.

Let $\e>0$ and let $\cc$ be a set of $k$-dimensional affine conditions of support size at most $s$. We may assume that each $(\phi,E)$ is such that there exists $x\in S^n$ with $\phi(x)\notin E$, since we can throw away all conditions that do not satisfy this property.

Now let $\{(\phi_1,E_1),\dots,(\phi_r,E_r)\}$ be a maximal subset of $\cc$ such that the maps $\phi_1,\dots,\phi_r$ are disjointly supported. Then the events $\phi_i(x)\in E_i$ are independent, and each holds with probability at most $1-2^{-s}$. Therefore, if $(1-2^{-s})^r<\e$, we can approximate $\cc$ by a single condition $(\phi,\emptyset)$. 

If $(1-2^{-s})^r\geq\e$, then for every $(\phi,E)\in\cc$ the support of $\phi$ intersects the union $Z$ of the supports of $\phi_1,\dots,\phi_r$. It follows that $\phi$ agrees with a map of support size at most $s-1$ on each set of the form $\{x:\forall j\in Z\ x_j=y_i\}$. By the inductive hypothesis and Lemma \ref{ball lemma} we are done. \end{proof}

\begin{proof}[Proof of Theorem \ref{Compression of polynomial conditions to a single condition}]
There are at most $(d+1)^k$ degree profiles that the polynomials can have, so by Lemma \ref{union lemma} we may assume that they all have the same degree profile. We may also throw away all conditions $(\phi, E)$ where $E = \F_p^k$. As promised in the introduction to this section, we now prove the result by induction on the degree profile, with respect to the colex order.

Let $\e>0$ and let $\cc$ be a set of $k$-dimensional polynomial conditions $(\phi,E)$ such that each polynomial $\phi$ has degree profile $(k_0,k_1,\dots,k_d)$. If the satisfying set of $\cc$ in $S^n$ has density less than $\e$, then we can approximate it with a single condition of the form $(\phi,\emptyset)$. Otherwise, Corollary \ref{bias implies low rank} gives us $R$ and $\tau$ such that for any $r$ polynomials $\phi_1,\dots,\phi_r$ involved in the conditions in $\cc$, if $r>R$ then there exist $a_1^*,\dots,a_r^*$ such that $\sum_{i=1}^ra_i^*\phi_i$ has rank at most $\tau$. (By this we mean it has degree $d'$ for some $d'\leq d$ and $(d',S)$-rank at most $\tau$.)

Let $\phi_1,\dots,\phi_r$ be a maximal set of polynomials coming from the conditions in $\cc$ for which that is \emph{not} the case, and suppose that $\phi_i:\F_p^n\to V_i$. Then for every $(\phi,E)\in\cc$ with $\phi:\F_p^n\to V$ there exist $a_1^*,\dots,a_r^*$ and $b^*$ with $a_i^*\in V_i^*$ and $b^*\in V^*\setminus\{0\}$ such that $\sum_{i=1}^ra_i^*\phi_i+b^*\phi$ has rank at most $\tau$. 

After a further application of Lemma \ref{union lemma}, we may assume that $(a_1^*,\dots,a_r^*,b^*)$ is the same for each $(\phi,E)\in\cc$. Let $\psi$ be the polynomial $\sum_{i=1}^ra_i^*\phi_i$, so $\psi+b^*\phi$ has rank at most $\tau$ for each $(\phi,E)\in\cc$.

Let us now fix $(\phi,E)\in\cc$. On each set $\{x\in S^n:\psi(x)=y\}$, the polynomials $b^*\phi$ and $\psi+b^*\phi-y$ agree, so $b^*\phi$ agrees with a polynomial of rank at most $\tau$. Let $V^*=V_0^*\oplus\dots\oplus V_d^*$ be a decomposition as described at the beginning of the section, so $\dim V_i^*=k_i$ and $w^*\phi$ has degree exactly $i$ for every $w^*\in V_i^*$. Let $b^*=b_0^*+b_1^*+\dots+b_d^*$ with $b_i^*\in V_i^*$.

Then if $b^*\phi$ has degree $d'$, it follows that $b_{d'}^*\ne 0$ and $b_i^*=0$ for every $i>d'$. Now pick $w\in V_{d'}$ such that $b_{d'}^*(w)=1$. Then we can write 
\[\phi(x)=b_{d'}^*\phi(x)w+(\phi(x)-b_{d'}^*\phi(x)w)\]
and the image of the polynomial map $\phi(x)-b_{d'}^*\phi(x)w$ lies in the kernel of $b_{d'}^*$.

Since $b^*\phi$ agrees with a polynomial of degree $d'$ and rank at most $\tau$, and since $(b_0^*+\dots+b_{d'-1}^*)\phi$ has degree less than $d'$, $b_{d'}^*\phi$ has rank at most $\tau+1$. We can therefore write it in the form $F(\g_1,\dots,\g_t)$ with $t\leq \tau+1$ and with $\g_1,\dots,\g_t$ all of degree less than $d'$. This allows us to replace the condition $(\phi,E)$ by a condition with a degree profile that comes earlier than $(k_0,k_1,\dots,k_d)$ and that agrees with $\phi$ on $\{x\in S^n:\psi(x)=y\}$. Let $\tilde\phi$ be the polynomial $\phi-(b_{d'}^*\phi)w$, regarded as taking values in $\ker(b_{d'}^*)$, and note that $\tilde\phi$ has degree profile $(k_0,k_1,\dots,k_{d'-1},k_{d'}-1,k_{d'+1},\dots,k_d)$. Let $\iota:\ker(b_{d'}^*)\to V$ be the inclusion map, and let $\theta:\ker(b_{d'}^*)\times\F_p^t\to V$ be the map $(u,a)\mapsto \iota u+aw$. Now let us take the polynomial $(\tilde\phi,\g_1,\dots,\g_t)$ and let $H\subset\ker(b_{d'}^*)\times\F_p^t$ be the set $\{(v,u):\theta(v,F(u))\in E\}$. 

Then $(\tilde\phi,\g_1,\dots,\g_t)(x)\in H$ if and only if $\iota\tilde\phi(x)+F(\g_1(x),\dots,\g_t(x))\in E$, which is true if and only if $\phi(x)-b_{d'}^*\phi(x)w+F(\g_1(x),\dots,\g_t(x))w\in E$, which is true if and only if $\phi(x)\in E$. Thus, the condition $(\phi,E)$ is equivalent to the condition $((\tilde\phi,\g_1,\dots,\g_t),H)$. But $\tilde\phi$ has degree profile $(k_0,k_1,\dots,k_{d'-1},k_{d'}-1,k_{d'+1},\dots,k_d)$, as already noted, and since each $\g_i$ has degree less than $d'$, the degree profile of $(\tilde\phi,\g_1,\dots,\g_t)$ precedes that of $\phi$ in colex order. Furthermore, since $t$ is bounded independently of $n$, all the degrees in the degree profile of $(\tilde\phi,\g_1,\dots,\g_t)$ are bounded independently of $n$. Hence we are done by the inductive hypothesis and Lemma \ref{ball lemma}, once we have established the base case.

Recall that for polynomials of degree 1, we define the rank to be the size of the support of the linearization. So the inductive argument just given allows us to reduce to the case where all the polynomials are linear and have bounded support size. This case is dealt with by Lemma \ref{bounded support linear case}.
\end{proof}

\section{Conclusion and open questions}\label{Section: Conclusion and open questions}

We finish by mentioning two questions to which we do not know the answers. One is the obvious one of improving the bounds in our approximation results. Even in the most basic approximation result, Proposition \ref{Compression in the basic case}, the question of the optimal bounds is not yet settled.

\begin{question}
Can we replace the bound $(p \log \epsilon^{-1})^p$ by $O_p(\log \epsilon^{-1})$ in Proposition \ref{Compression in the basic case}?
\end{question}

If this is indeed possible, then we may ask whether we can also take bounds $O(\log \epsilon^{-1})$ in Theorem \ref{main approximation theorem}, Theorem \ref{Compression in the groups case for arbitrary alphabets}, and Theorem \ref{Compression of polynomial conditions to a single condition}. In the case of Theorem \ref{Compression of polynomial conditions to a single condition} it would probably help to first have linear bounds $H_{p,d,S}(\epsilon) = O_{p,d,S}(\log \epsilon^{-1})$ in Theorem \ref{Equidistribution of polynomials} on the equidistribution of polynomials.

In a second direction, our first approximation result for linear forms $\F_p^n \rightarrow \F_p$, Proposition \ref{main approximation theorem}, involves internal approximation. We can also ask for an external approximation, where it also does not seem obviously unreasonable to require the conditions to belong to the original family of conditions. (In the case of internal approximations, it is not possible to add this additional requirement, as it would imply that the set of $x \in S^n$ satisfying the approximating set of conditions is exactly the same as that satisfying the original set of conditions, which is not always possible, as can be seen with an example we gave earlier, where we take $S$ to be $\{0,1\}$ and the family of linear forms to be $\{(x_1 + x_i, \{0,1\}): i \in \lbrack 2, n \rbrack\}$.)

\begin{conjecture} Let $p$ be a prime, let $S$ be a non-empty subset of $\F_p$, and let $\epsilon > 0$. There there exists a constant $A=A(p,S,\e)$ such that for every set $\mathcal F$ of mod-$p$ conditions on $S^n$ there is a subset $\mathcal G$ of $\mathcal F$ such that the satisfying set of $\mathcal G$ $\e$-externally approximates that of $\mathcal F$ (meaning that the density of points that satisfy $\mathcal G$ but not $\mathcal F$ is at most $\e$). 
\end{conjecture}

\end{document}